\pdfoutput=1
\documentclass{scrartcl}
\title{The noncommutative geometry \\ of frame bundles}
\author{Stefan Wagner}
\date{}

\KOMAoptions{parskip=half,paper=a4,abstract=on}

\usepackage{amsmath, amsthm, amssymb, bbm, mathtools, nicefrac}
\usepackage{xifthen}
\usepackage[english]{babel}
\usepackage[T1]{fontenc}
\usepackage[utf8]{inputenc}
\usepackage{lmodern}
\usepackage{tabularx}
\usepackage{marvosym}
\usepackage{wasysym}
\usepackage{graphicx}
\usepackage[bookmarks,bookmarksdepth=1]{hyperref}
\usepackage[dvipsnames]{xcolor}
\usepackage{eurosym}
\usepackage{microtype}		
\usepackage[english]{babel}
\usepackage{enumitem}
	\newlist{equivalence}{enumerate}{1}
	\setlist[equivalence]{label=(\alph*)}
\usepackage{xspace}
\usepackage[final]{pdfpages}
\usepackage{url}
\usepackage{tikz}
\usepackage{tabu}
\usetikzlibrary{topaths}
\tikzstyle{every picture}+=[remember picture,inner xsep=0,inner ysep=0.25ex]
\usepackage{booktabs}
\usepackage{aligned-overset}
\usepackage{fixme}
\newtheoremstyle{plain}
  {\topsep}   
  {\topsep}   
  {\itshape}  
  {0pt}       
  {\bfseries\sffamily} 
  {.}         
  {5pt plus 1pt minus 1pt} 
  {}          

\newtheoremstyle{definition}
  {\topsep}   
  {\topsep}   
  {\normalfont}  
  {0pt}       
  {\bfseries\sffamily} 
  {.}         
  {5pt plus 1pt minus 1pt} 
  {}          
  
\newtheoremstyle{remark}
  {0.5\topsep}   
  {0.5\topsep}   
  {\normalfont}  
  {0pt}       
  {\itshape} 
  {.}         
  {5pt plus 1pt minus 1pt} 
  {}          

\theoremstyle{plain}
	\newtheorem{theorem}{Theorem}[section]
	\newtheorem{lemma}[theorem]{Lemma}
	\newtheorem{corollary}[theorem]{Corollary}
	
\theoremstyle{definition}
	\newtheorem{defn}[theorem]{Definition}
\theoremstyle{remark}
	\newtheorem{remark}[theorem]{Remark}

\makeatletter
\newcommand{\setword}[2]{%
  \phantomsection
  #1\def\@currentlabel{\unexpanded{#1}}\label{#2}%
}

\newdimen\bibindent
\newdimen\bibspacing
\makeatletter
\newcommand{\raisemath}[1]{\mathpalette{\raisem@th{#1}}}
\newcommand{\raisem@th}[3]{\raisebox{#1}{$#2#3$}}
\makeatother

\newcommand*{\Z}{\mathbb{Z}}		
\newcommand*{\R}{\mathbb{R}}		
\newcommand*{\C}{\mathbb{C}}		
\newcommand*{\tensor}{\otimes}		

\DeclarePairedDelimiter{\scal}{\langle}{\rangle}	

\DeclareMathOperator{\id}{id}		
\DeclareMathOperator{\SU}{SU}
\DeclareMathOperator{\SO}{SO}
\DeclareMathOperator{\GL}{GL}

\DeclareMathOperator{\ran}{ran}

\newcommand{\cf}{\mbox{cf.}\xspace}				
\newcommand{\eg}{\mbox{e.\,g.}\xspace}			
\newcommand*{\etal}{\mbox{et\,al.}\xspace}		
\newcommand*{\ie}{\mbox{i.\,e.}\xspace}			
\newcommand*{\Star}{$^*$\nobreakdash}

\newcommand*{\hilb}{\mathfrak}		
\newcommand*{\alg}{\mathcal}		
\newcommand*{\hH}{\hilb H}			
\newcommand*{\one}{1}				
\newcommand*{\aA}{\alg{A}}			
\newcommand*{\aB}{\alg{B}} 			
\newcommand*{\End}{\mathcal L}

\newcommand*{\Cont}{C}

\DeclareMathOperator{\Irrep}{Irr} 			
\newcommand*{\acts}{\,.\,}
\DeclarePairedDelimiterXPP{\lprod}[2]
	{\,_{#1}}						
	{\langle}						
	{\rangle}						
	{}								
	{#2}							
\DeclarePairedDelimiterX{\rprod}[2]{\langle}{\rangle_{#1}}{#2}
\DeclarePairedDelimiterX{\ketbra}[2]{\lvert}{\rvert}{#1 \delimsize\rangle \delimsize\langle #2}
\DeclarePairedDelimiterX{\braket}[2]{\langle}{\rangle}{#1 \delimsize\vert #2}
\DeclarePairedDelimiterX{\matrixel}[3]{\langle}{\rangle}{#1 \delimsize\vert #2 \delimsize\vert #3}
\DeclareMathOperator{\Ext}{Ext}
\DeclareMathOperator{\Pic}{Pic}
\DeclareMathOperator{\corr}{Corr}
\DeclareMathOperator{\rep}{Rep}
\DeclareMathOperator{\Fr}{Fr}

\DeclareMathOperator{\Iso}{Iso}
\DeclareMathOperator{\Hom}{Hom}

\newcommand{\doubleitem}{%
  \begingroup
  \stepcounter{enumi}%
  \edef\tmp{\theenumi,}%
  \stepcounter{enumi}
  \edef\tmp{\endgroup\noexpand\item[\tmp\labelenumi]}%
  \tmp}
  
 \newcommand{\tripleitem}{%
 \begingroup \stepcounter{enumi}%
 \edef\tmp{\theenumi,}%
 \stepcounter{enumi} 
 \edef\tmpt{\theenumi,} 
 \stepcounter{enumi} 
 \edef\tmp{\endgroup\noexpand\item[\tmp\tmpt\labelenumi]}%
 \tmp}

\hypersetup{colorlinks=true,
  linkcolor = black,
  citecolor = BurntOrange, 
  urlcolor =  BurntOrange,
}

\begin{document}

\author{
	Stefan Wagner \thanks{
		Blekinge Tekniska H\"ogskola,
		\href{mailto:stefan.wagner@bth.se}{\nolinkurl{stefan.wagner@bth.se}}
	}
}

\sloppy
\maketitle

\begin{abstract}
	\noindent
	We apply ourselves to the noncommutative geometry of frame bundles by showing that each C\Star-algebraic noncommutative principal $\SO(n)$-bundle is, up to isomorphism, uniquely determined by its associated noncommutative vector bundle with respect to the standard representation of $\SO(n)$.
For this, we provide a construction procedure, via unitary tensor functors, that for a certain type of correspondence, let's say~$M$, attaches a free C\Star-dynamical system $(\aA_M,\SO(n),\alpha_M)$ with the property that its associated noncommutative vector bundle with respect to the standard representation of $\SO(n)$ is isomorphic to~$M$.

	\vspace*{0.3cm}

	
\end{abstract}

\listoffixmes

\section{Introduction}\label{sec:intro}

\enlargethispage{\baselineskip}
Vector bundles in classical geometry typically arise as objects associated with something more profound, a principal bundle.
In particular, each vector bundle $E$ with fibre~$V$ is naturally associated with a principal $\GL(V)$-bundle, called the \emph{frame bundle} of $E$ and denoted by $\Fr(E)$.
Frame bundles thus constitute a key tool for studying vector bundles.
Indeed, given a vector bundle $E$, its frame bundle can be utilized in order to attach to~$E$ several new vector bundles in a functorial manner.
Furthermore, a connection on $\Fr(E)$ induces covariant derivatives on all associated bundles in a coherent way, leading to many important geometric constructions.
This is the situation in Riemannian geometry where, for a Riemannian manifold $X$, the Levi-Civita connection on $\Fr(TX)$ induces a covariant derivative on the tensor fields, leading, for instance, to the Riemannian curvature of~$X$.
Another instance is Riemannian spin geometry where, for a spin manifold $X$, a ``spin connection'' on $\Fr(TX)$ induces a covariant derivative on the spinor bundle, leading to the Dirac  and Laplace operator on the the spinor bundle.

The noncommutative geometry of frame bundles, however, has not been studied conclusively, although the notion of a noncommutative principal bundle is certainly available, accompanied by a natural procedure of associating noncommutative vector bundles (see, \eg, \cite{BaCoHa15,Ell00,SchWa17} and ref.~therein). 
To the best of our knowledge, the first and only systematic treatment of a quantized notion of ``frame bundle'' seems to be due to S.~Majid~\cite{Maj99,Maj02,Maj05} (see also~\cite[Sec.~5.6]{BeMa20} for a cohesive presentation).
For the sake of expedience, we briefly recall his main idea.
Let $B$ be an algebra equipped with a differential structure $(\Omega^1,d)$ which plays the role of the differential 1-forms of an ordinary manifold.
A \emph{framing} of $B$ is then essentially a Hopf-Galois-algebraic noncommutative principal bundle over $B$ that recovers $\Omega^1$ in the sense that there exists a suitable associated noncommutative vector bundle isomorphic to $\Omega^1$.
Note that the structure quantum group need not be fixed in this approach as one might have several different candidates.

In this paper we sort of complement Majid's approach within the C\Star-algebraic framework of noncommutative principal bundles.
More precisely, for a unital C\Star-algebra $\aB$ and a certain type of correspondence over $\aB$ which plays the role of the vector bundle associated with an ordinary frame bundle with respect to the standard representation $\pi$ of $\SO(n)$, let's say $M$, we provide a construction procedure for a free C\Star-dynamical system $(\aA_M,\SO(n),\alpha_M)$ with fixed point algebra $\aB$ and the property that its associated noncommutative vector bundle with respect to the standard representation of $\SO(n)$ is isomorphic to~$M$.
We stress that such a construction procedure already exists for the special case of the compact Abelian group $\SO(2)$ (\cf~\cite[Sec.~4]{SchWa15}), and therefore we assume from now on that $n \geq 3$.

Although our construction procedure is of topological nature, we hope that parts of it extend in such a way that additional geometrical information is comprised.
In particular, we hope to be able to construct new and interesting characteristic classes in future work.
On a related note, we wish to mention that this article is part of a larger program with the purpose to give a novel perspective on noncommutative Riemannian spin geometry by systematically developing and studying the key constructions and ideas of Riemannian spin geometry within the C\Star-algebraic framework of noncommutative principal bundles.
Finally, we would like to point out that with little effort our arguments and results extend to semisimple Lie groups admitting a faithful irreducible representation.

\subsection*{Organization of the article}

Here is an outline of this paper.
After providing the necessary preliminaries and notation, we present in Section~\ref{sec:NCFB}, the main body of this article, our approach to noncommutative frame bundles.
More precisely, for a unital C\Star-algebra $\aB$ and the standard representation $\pi$ of $\SO(n)$, $n \geq 3$, we introduce the 
central notion of this work which is concerned with correspondences over $\aB$ being ``\emph{tensorial of type $\pi$}'' (Definition~\ref{def:pi-corr} and Definition~\ref{def:pi-corr_tensor}).
For such a correspondence, let's say $M$, we are able to provide a construction procedure for a unitary tensor functor from a small tensor subcategory $\mathcal{T} \subseteq \rep(\SO(n))$ containing representatives of $\Irrep(\SO(n))$ to the tensor category $\corr(\aB)$ of correspondences over $\aB$.
The construction procedure relies on the representation theory of the semisimple compact Lie group $\SO(n)$, and it naturally splits into the following two main steps:
\begin{enumerate}
	\item
		We construct a small tensor subcategory $\mathcal{T} \subseteq \rep(\SO(n))$ containing representatives of $\Irrep(\SO(n))$ together with a linear functor $\Gamma_M: \mathcal{T} \to \corr(\aB)$ satisfying $\Gamma_M\big(\pi^{\otimes k}\big) = M^{\tensor k}$ for all integers $k \geq 0$ and $\Gamma_M(T)^* = \Gamma_M(T^*)$ for all morphisms $T$ in $\mathcal{T}$ (Corollary~\ref{cor:linfun}).
	\item
		We construct natural, $\aB$-bilinear, and unitary maps $m_M(\sigma,\tau) : \Gamma_M(\sigma) \tensor_\aB \Gamma_M(\tau) \to \Gamma_M(\sigma \tensor \tau)$ for all $\sigma,\tau \in\mathcal{T}$ (Corollary~\ref{cor:unitary_S}), and show that they satisfy a certain associativity condition (Lemma~\ref{lem:asso}).
\end{enumerate}
Taken together, the results amount to a unitary tensor functor $\mathcal{T} \to \corr(\aB)$ (Theorem~\ref{thm:unitenfun}) which in turn yields a free C\Star-dynamical system $(\aA_M,\SO(n),\alpha_M)$ with fixed point algebra $\aB$ and  $\Gamma_{\aA_M}(\pi) \cong M$, the \emph{noncommutative frame bundle associated with $M$} (Corollary~\ref{cor:main1}).
For expediency's sake, we discuss the main steps of the construction of $(\aA_M,\SO(n),\alpha_M)$ to Appendix~\ref{sec:construction}.
Finally, we present a classification result that extends the classical correspondence between frame bundles and their associated vector bundles (Corollary~\ref{cor:main2}).
Section~\ref{sex:examples} is devoted to examples.
Indeed, let $V$ be the representation space of $\pi$.
In Section~\ref{ex:trivial}, we show that $\aB \tensor V$ is tensorial of type~$\pi$.
In Section~\ref{ex:VB}, we demonstrate that each locally trivial hermitian vector bundle over a compact space with typical fibre $V$ and structure group $\SO(n)$ is tensorial of type~$\pi$, thus recovering the classical setting of frame bundles.
In Section~\ref{ex:quantumprojective} and Section~\ref{ex:cuntz}, we introduce new, as of yet unknown, free C\Star-dynamical systems with structure group $\SO(3)$, the \emph{quantum projective 7-space} and the (even part of the) Cuntz algebra $\mathcal{O}_2$, and hence we get two more examples of tensorial correspondences of type $\pi$ by looking at their associated noncommutative vector bundles with respect to $\pi$.
Last but not least, in Appendix~\ref{sec:MC}, we briefly treat the special case of $\SO(2)$ for the sake of completeness.

\section{Preliminaries and notation}\label{sec:prelim}

Our study deals with the noncommutative geometry of frame bundles.
This preliminary section exhibits the most fundamental definitions and notations in use.

At first we provide some standard references.
For a thorough treatment of Lie theory and representation theory we refer to the remarkable work~\cite{GoWa09} by Goodman and Wallach (see also~\cite{BroDi85}).
For a comprehensive introduction to the theory of fibre bundles, especially principal bundles and (their associated) vector bundles, we refer to Husem\"oller's book~\cite{Huse94} and the influential exposition~\cite{Kob69I} by Nomizu and Kobayashi. 
For a recent account of the theory of 
Hilbert module structures we refer to the excellent volume~\cite{Rae98} by Raeburn and Williams and the memoirs~\cite{EKQR06} by Echterhoff \etal
Our standard references for the theory of operator algebras are the opuses~\cite{BB06,Ped18} by Blackadar and Pedersen, respectively. 
We also use a categorical description of noncommutative principal bundles, and for the necessary background on category theory we refer to the monographs~\cite{EtGENiOs15,Mi65,NeTu10}.

\paragraph*{About groups}
Let $G$ be a compact group.
All representations of $G$ are assumed to be finite-dimensional and unitary unless mentioned otherwise.
We denote a representation $\sigma: G \to \mathcal{U}(V_\sigma)$ by the pair $(\sigma,V_\sigma)$ or simply by $\sigma$.
In particular, we let $\one$ stand for the trivial representation when no ambiguity is possible.
We write $\rep(G)$ for its C\Star-tensor category of representations and $\Irrep(G)$ for the set of equivalence classes of irreducible representations.
By abuse of notation, we also use the symbol $\sigma$ to denote elements of $\Irrep(G)$ and choose a representative representation $(\sigma,V_\sigma)$ for $\sigma \in \Irrep(G)$ when needed.

One of the key ingredients for our construction procedure in Section~\ref{sec:NCFB} is the following result on irreducible representations of the semisimple compact Lie group $\SO(n)$, $n \geq 3$.

\begin{corollary}[{See, \eg,~\cite[Thm.~5.5.21]{GoWa09}}]\label{cor:tensorprodSO(n)}
	Let $\pi$ be the standard representation of $\SO(n)$,~$n \geq 3$.
	Each irreducible representation of $\SO(n)$ occurs as a subrepresentation of some tensor product representation $\pi^{\otimes k}$, $k \geq 0$ (with $\pi^{\otimes 0} = \one$).
\end{corollary}

\paragraph*{About frames bundles}
Let $X$ be a locally compact space and let $q:E \to X$ be a locally trivial (real or complex) vector bundle over $X$ with typical fibre $V$. 
The \emph{frame bundle}
\begin{align*}
	\Fr(E):=\bigcup_{x \in X} \Iso\left(V,E_x\right),
	&&
	E_x := q^{-1}(\{x\}),
\end{align*}
carries the structure of a principal $\mathrm{GL}(V)$-bundle over $X$ with respect to the canonical right action of $\mathrm{GL}(V)$ on $\Fr(E)$.
The associated vector bundle $\Fr(E) \times_{\mathrm{GL}(V)} V$ with respect to the standard representation $(\pi,V)$ of $\mathrm{GL}(V)$ recovers $E$, \ie, $\Fr(E) \times_{\mathrm{GL}(V)} V \cong E$ as vector bundles over $X$.
For simplicity of notation, we use the same symbol $\Fr(E)$ to denote any reduction of the frame bundle.

\paragraph*{About Hilbert modules}
Let $\aB$ be a unital C\Star-algebra.  
A~\emph{correspondence} over $\aB$ is a $\aB$-bimodule $M$ equipped with an inner product $\rprod{\aB}{\cdot, \cdot}: M \times M \to \aB$ turning it into a right Hilbert $\aB$-module such that the left action of $\aB$ on $M$ is through adjointable operators. 
Given two correspondences $M$ and $N$ over $\aB$, we write $M \tensor_\aB N$ for their tensor product on which the inner product is determined by $\scal{x_1 \tensor y_1, \; x_2 \tensor y_2}_\aB = \scal{y_1, \scal{x_1,x_2}_\aB \,. y_2}_\aB$ for all $x_1, x_2 \in M$ and $y_1, y_2 \in N$.
We are also concerned with multiple tensor products.
For a correspondence $M$ over $\aB$ and a non-negative integer $k$ we let $M^{\otimes k}$ stand for the $k$-fold tensor product of $M$ with itself (with $M^{\otimes 0} = \aB$).
We use the symbol $\corr(\aB)$ to denote the C\Star-tensor category of correspondences over~$\aB$.

\paragraph*{\boldmath About C$^*$-dynamical systems}\label{sec:Cstar}
Let $\aA$ be a unital C\Star-algebra and let $G$ be a compact group that acts strongly continuously on $\aA$ by \Star-automorphisms $\alpha_g:\aA \to \aA$, $g \in G$.
Throughout this article we call such data a \emph{C\Star-dynamical system}, denote it briefly by $(\aA,G,\alpha)$, and typically write $\aB := \aA^G$ for its fixed point algebra. 

\enlargethispage{\baselineskip}
\begin{remark}
	Like every, possibly infinite, continuous representation of a compact group, the algebra $\aA$ can be decomposed into its isotypic components which amounts to saying that their algebraic direct sum forms a dense \Star-subalgebra of~$\aA$ (see, \eg, \cite[Thm.~4.22]{HoMo06}). 
\end{remark}

We also deal to a large extend with the associated spaces
\begin{align}\label{eq:NCVB}
	\Gamma_\aA(\sigma) := \{ x \in \aA \tensor V_\sigma : (\forall g\in G) \, (\alpha_g \tensor \sigma_g)(x) = x\}
\end{align}
for all objects $\sigma$ in $\rep(G)$, each of which is naturally a correspondence over $\aB$ with respect to the canonical left and right actions and the restriction of the right $\aA$-valued inner product on $\aA \tensor V_\sigma$ determined by $\rprod{\aA}{a \tensor v, b \tensor w} := \scal{v,w} a^*b$ for all $a,b \in \aA$ and $v,w \in V_\sigma$. 
Most notably, the linear functor $\Gamma_\aA: \rep(G) \to \corr(\aB)$, defined for objects by $\Gamma_\aA(\sigma)$ and for morphisms by $\Gamma_\aA(T) := \one_\aA \tensor T$, together with the natural $\aB$-bilinear isometries $m_\aA(\sigma,\tau) : \Gamma_\aA(\sigma) \tensor_\aB \Gamma_\aA(\tau) \to \Gamma_\aA(\sigma \tensor \tau)$, $x \tensor y \mapsto x_{12} y_{13}$ for all objects $\sigma,\tau$ in $\rep(G)$ constitute a weak unitary tensor functor $\rep(G) \to \corr(\aB)$ which allows to reconstruct the C\Star-dynamical system $(\aA,G,\alpha)$ up to isomorphism (see~\cite[Sec.~2]{Ne13}). 
%

\paragraph*{About freeness}\label{sec:free}
A C\Star-dynamical system $(\aA,G,\alpha)$ is called \emph{free} if the \emph{Ellwood map} 
\begin{align*}
	\Phi: \aA \tensor_{\text{alg}} \aA \rightarrow C(G,\aA),
	&&
	\Phi(x\tensor y)(g):=x \alpha_g(y)
\end{align*}
has dense range with respect to the canonical C\Star-norm on $\Cont(G,\aA)$.
This requirement was originally introduced for actions of quantum groups on C\Star-algebras by Ellwood~\cite{Ell00} and is known to be equivalent to Rieffel's saturatedness~\cite{Rieffel91} and the Peter-Weyl-Galois condition~\cite{BaCoHa15}. 
By \cite[Prop.~7.1.12 \& Thm.~7.2.6]{Phi87}, we can assert that a continuous action $r:P\times G\rightarrow P$ of a compact group $G$ on a compact space $P$ is free in the classical sense if and only if the induced C\Star-dynamical system $(C(P),G,\alpha)$ with $\alpha_g(f) := f \circ r_g$ is free in the sense of Ellwood. 
Free C\Star-dynamical systems thus provide a natural framework for noncommutative principal bundles, and in this context the spaces in~\eqref{eq:NCVB} play the role of the associated vector bundles.
In particular, they are finitely generated and projective as right $\aB$-modules (see,~\eg,~\cite[Thm.~1.2]{CoYa13a}).

Yet another crucial result, which plays an important role in the construction of noncommutative frame bundles, is that free C\Star-dynamical systems are in bijective correspondence with so-called unitary tensor functors (\cf~\cite[Sec.~5]{SchWa17}), whose definition we now recall for the convenience of the reader.

\begin{defn}\label{defn:wutfun}
	Let $\aB$ be a unital C\Star-algebra and let $G$ be a compact group.
	A \emph{unitary tensor functor} $\rep(G) \to \corr(\aB)$ is a linear functor $\Gamma: \rep(G) \to \corr(\aB)$ together with natural $\aB$-bilinear unitary maps $m(\sigma,\tau) : \Gamma(\sigma) \tensor_\aB \Gamma(\tau) \to \Gamma(\sigma \tensor \tau)$ for all $\sigma,\tau \in \rep(G)$ such that the following conditions are satisfied:
	\begin{enumerate}[nosep]
	\item[(i)]
		$\Gamma(\one) = \aB$, and for each object $\sigma$ in $\rep(G)$ the map $m(\one,\sigma)$ maps $b \tensor x$ to $b \acts x$ and, similarly, $m(\sigma,\one)$ maps $x \tensor b$ to $x \acts b$ for all $b \in \aB$ and $x \in \Gamma(\sigma)$.
	\item[(ii)]
		$\Gamma(T)^* = \Gamma(T^*)$ for all morphisms $T$ in $\rep(G)$.
	\item[(iii)]
		$m(\sigma, \tau \tensor \rho) \left(\id \tensor_\aB \, m(\tau, \rho)\right) = m(\sigma \tensor \tau, \rho) \left(m(\sigma, \tau) \tensor_\aB \id\right)$ for all objects $\sigma, \tau, \rho$ in $\rep(G)$.
%
	\end{enumerate}
\end{defn}

\enlargethispage{\baselineskip}
\begin{remark}\label{rem:subcat}
	For the construction of a free C\Star-dynamical system from a unitary tensor functor $\rep(G) \to \corr(\aB)$ one only needs a small C\Star-tensor subcategory $\mathcal{T} \subseteq \rep(G)$ containing representatives of $\Irrep(G)$ (\cf~\cite[Thm.~2.3]{Ne13}).
\end{remark}

\pagebreak[3]
\section{Noncommutative frame bundles}\label{sec:NCFB}

In this section we present our approach to noncommutative frame bundles.
More precisely, for a unital C\Star-algebra $\aB$ and a certain type of correspondence over $\aB$ which plays the role of the vector bundle associated with an ordinary frame bundle with respect to the standard representation $\pi$ of $\SO(n)$, $n \geq 3$, let's say $M$, we provide a construction procedure for a free C\Star-dynamical system $(\aA_M,\SO(n),\alpha_M)$ with fixed point algebra $\aB$ and $\Gamma_{\aA_M}(\pi) \cong M$, the \emph{noncommutative frame bundle associated with $M$} (see Appendix~\ref{sec:MC} for the case $n=2$).
The main idea is to put together a unitary tensor functor from a small tensor subcategory $\mathcal{T} \subseteq \rep(\SO(n))$ containing representatives of $\Irrep(\SO(n))$ to $\corr(\aB)$ (\cf~Remark~\ref{rem:subcat}). 

We begin by introducing the main notion of this article:

\begin{defn}\label{def:pi-corr}
Let $\aB$ be a unital C\Star-algebra and let $\pi$ be the standard representation of $\SO(n)$, $n \geq 3$.
We say that \emph{a correspondence $M$ over $\aB$ is of type~$\pi$} if there exist injective linear maps $\varphi_{k,l} : C_{k,l}:= \Hom_{\SO(n)}\big(V_\pi^{\tensor k},V_\pi^{\tensor l}\big) \to \End\big(M^{\tensor k},M^{\tensor l}\big)$ for all integers $k,l \geq 0$ such that the following compatibility conditions are satisfied:
	\begin{enumerate}[nosep]
	\item[(C)]
		$\varphi_{l,m}(T') \varphi_{k,l}(T) = \varphi_{k,m}(T'T)$ for all $k,l,m \geq 0$, $T\in C_{k,l}$,  and $T' \in C_{l,m}$.
	\item[(A)]
		$\varphi_{k,l}(T)^* = \varphi_{l,k}(T^*)$ for all $k,l \geq 0$ and $T\in C_{k,l}$.
	\item[(U)]
		$\varphi_{k,k}(\id) = \id$ for all $k \geq 0$.
	\end{enumerate}
\end{defn}
From now on, let $\aB$ be a unital C\Star-algebra, let $\pi$ be the standard representation of $\SO(n)$, $n \geq 3$, and let $M$ be a correspondence over $\aB$ of type~$\pi$.
Also, let $V := V_\pi$ for brevity.
As a first step we construct a small tensor subcategory $\mathcal{T} \subseteq \rep(G)$ containing representatives of $\Irrep(\SO(n))$ together with a linear functor $\Gamma_M: \mathcal{T} \to \corr(\aB)$.
For this purpose, we choose for each $\sigma \in \Irrep(\SO(n))$ a representative 
$(\sigma,V_\sigma)$ that is a subrepresentation of some tensor product representation $\big(\pi^{\otimes k},V^{\otimes k}\big)$, $k \geq 0$ (\cf~Corollary~\ref{cor:tensorprodSO(n)}).
In particular, for $\one \in \Irrep(\SO(n))$ we choose the trivial representation $(\one,\C)$.
Furthermore, we consider the full subcategory $\mathcal{S} \subseteq \rep(G)$ whose objects consists of all finite tensor product representations generated by the family $(\sigma,V_\sigma)$, $\sigma \in \Irrep(\SO(n))$.

Let $(\sigma,V_\sigma)$ be an object in $\mathcal{S}$.
By construction, there exists an integer $k \geq 0$ such that $(\sigma,V_\sigma)$ is a subrepresentation of $\big(\pi^{\otimes k},V^{\otimes k}\big)$.
Let $P_\sigma$ be the orthogonal projection of $V^{\otimes k}$ onto~$V_\sigma$.
Clearly, $P_\sigma \in C_{k,k}$, and hence $\varphi_{k,k}(P_\sigma)$ acts as an adjointable operator on $M^{\tensor k}$.
Moreover, Conditions~\hyperref[def:pi-corr]{(C) and (A)} combined imply that $\varphi_{k,k}(P_\sigma)$ is a projection, and from this it may be concluded that
\begin{align}\label{eq:Gamma_M}
	\Gamma_M(\sigma) := \varphi_{k,k}(P_\sigma)\left( M^{\tensor k} \right)
\end{align}
is a correspondence over $\aB$.
We thus have a correspondence over $\aB$ available for each object $(\sigma,V_\sigma)$ in $\mathcal{S}$.
Note that $\Gamma_M\big(\pi^{\otimes k}\big) = M^{\tensor k}$ for all integers $k \geq 0$ by Condition~\hyperref[def:pi-corr]{(U)}.

Next, let $(\sigma,V_\sigma)$ and $(\tau,V_\tau)$ be objects in $\mathcal{S}$ and let $T:V_\sigma \to V_\tau$ be a morphism.
Our objective is to relate the correspondences $\Gamma_M(\sigma)$ and $\Gamma_M(\tau)$ by means of a morphism $\Gamma_M(T)$.
To this end, let $k,l \geq 0$ be integers such that $(\sigma,V_\sigma)$ and $(\tau,V_\tau)$ are subrepresentations of $\big(\pi^{\otimes k},V^{\otimes k}\big)$ and  $\big(\pi^{\otimes l},V^{\otimes l}\big)$, respectively, and let $P_\sigma$ and $P_\tau$ be the orthogonal projections of $V^{\otimes k}$ onto~$V_\sigma$ and $V^{\otimes l}$ onto~$V_\tau$, respectively.
We define a map $W_T : V^{\otimes k} \to V^{\otimes l}$~by
\begin{align*}
	W_T(x) 
	:= 		
	\begin{cases}
		T(x) & \text{for $x \in V_\sigma$},
		\\
		0 & \text{for $x \in V_\sigma^\perp \subseteq V^{\otimes k}$}.
	 \end{cases}
\end{align*}
Note that $W_T^* = W_{T^*}$.
Moreover, it is easily seen that $W_T \in C_{k,l}$, and so we may look at its image under the map $\varphi_{k,l}$.
In fact, we have
\begin{align*}
	\varphi_{l,l}(P_\tau) \varphi_{k,l}(W_T) 
	\overset{\hyperref[def:pi-corr]{(\text{C})}}{=}
	\varphi_{k,l}(P_\tau W_T) = \varphi_{k,l}(W_T)
\end{align*}
which implies that $\ran(\varphi_{k,l}(W_T)) \subseteq \Gamma_M(\tau)$.
With this at hand, we put
\begin{align*}
	\Gamma_M(T) := \varphi_{k,l}(W_T) \restriction_{\Gamma_M(\sigma)}^{\Gamma_M(\tau)} : \Gamma_M(\sigma) \to \Gamma_M(\tau).
\end{align*}
It is worth noting that $W_{P_\sigma} = P_\sigma$, therefore that $\Gamma_M(P_\sigma) = \varphi_{k,k}(P_\sigma)$, and finally that $\ran(\Gamma_M(P_\sigma)) = \Gamma_M(\sigma)$.
Furthermore, we see at once that $\Gamma_M(\id_{V_\sigma}) = \id_{\Gamma_M(\sigma)}$.

By attentively following the above construction and applying the assumptions in Definition~\ref{def:pi-corr}, we immediately~get:

\begin{lemma}\label{lem:lin}
	For objects $(\sigma,V_\sigma)$, $(\tau,V_\tau)$, and $ (\rho,V_\rho)$ in $\mathcal{S}$ the following assertions hold:
	\begin{enumerate}
	\item
		$\Gamma_M(T+cT') = \Gamma_M(T) + c\Gamma_M(T')$ for all morphisms $T,T':V_\sigma \to V_\tau$ and $c \in \C$.
	\item
		$\Gamma_M(T'T) = \Gamma_M(T') \Gamma_M(T)$ for all morphisms $T:V_\sigma \to V_\tau$ and $T':V_\tau \to V_\rho$.
	\end{enumerate}
\end{lemma}

\begin{lemma}\label{lem:*-inv}
	 Let $(\sigma,V_\sigma)$ and $(\tau,V_\tau)$ be objects in $\mathcal{S}$ and let $T:V_\sigma \to V_\tau$ be a morphism.
	 Then $\Gamma_M(T)^* = \Gamma_M(T^*)$.
\end{lemma}
\begin{proof}
	Let $x \in \Gamma_M(\sigma)$ and let $y \in \Gamma_M(\tau)$.
	Since $\varphi_{k,l}(W_T)^* \overset{\hyperref[def:pi-corr]{(\text{A})}}{=} \varphi_{l,k}(W_T^*) = \varphi_{l,k}(W_{T^*})$, we conclude that
	\begin{gather*}
		\rprod{\aB}{\Gamma_M(T)(x),y} 
		=
		\rprod{\aB}{\varphi_{k,l}(W_T)(x),y} 
		=
		\rprod{\aB}{x,\varphi_{k,l}(W_T)^*(y)} 
		\\
		=
		\rprod{\aB}{x,\varphi_{l,k}(W_{T^*})(y)} 
		=
		\rprod{\aB}{x,\Gamma_M(T^*)(y)}.
		\qedhere
	\end{gather*}
\end{proof}

We proceed by looking at the full subcategory $\mathcal{T} \subseteq \rep(\SO(n))$ whose objects are finite direct sums of objects in $\mathcal{S}$.
It is clear that $\mathcal{T}$ is a small tensor subcategory of $\rep(\SO(n))$.
Furthermore, for an object $\sigma = \sigma_1 \oplus \dots \oplus \sigma_m$ in $\mathcal{T}$ we define 
\begin{align*}
	\Gamma_M(\sigma) := \Gamma_M(\sigma_1) \oplus \cdots \oplus \Gamma_M(\sigma_m).
\end{align*}
Also, for two objects $\sigma = \sigma_1 \oplus \dots \oplus \sigma_r$ and $\tau = \tau_1 \oplus \dots \oplus \tau_s$ in $\mathcal{T}$ we note that
\begin{align*}
	\Hom_G(V_\sigma,V_\tau) = \bigoplus_{i,j} \Hom_G(V_{\sigma_i},V_{\tau_j}),
\end{align*}
and hence each morphism $T:V_\sigma \to V_\tau$ can be uniquely written as a matrix with entries in $\Hom_G(V_{\sigma_i},V_{\tau_j})$ for all eligible pairs $i,j$.
Given such a morphism $T=(T_{ij})$, we put 
\begin{align*}
	\Gamma_M(T) := \left(\Gamma_M(T_{ij})\right)
\end{align*}
With these definitions, it is easily checked that Lemma~\ref{lem:lin} and Lemma~\ref{lem:*-inv} extend to objects in $\mathcal{T}$ and morphisms between them.
Summarizing, we have thus shown:

\begin{corollary}\label{cor:linfun}
	$\mathcal{T}$ is a small tensor subcategory of $\rep(SO(n))$ containing representatives of $\Irrep(\SO(n))$.
	Furthermore, the map $\Gamma_M: \mathcal{T} \to \corr(\aB)$, defined for objects by $\Gamma_M(\sigma)$ and for morphisms by $\Gamma_M(T)$, is a linear functor such that $\Gamma_M\big(\pi^{\otimes k}\big) = M^{\tensor k}$ for all integers $k \geq 0$ and $\Gamma_M(T)^* = \Gamma_M(T^*)$ for all morphisms $T$ in $\mathcal{T}$.
\end{corollary}

Having completed the first task, we now turn to the construction of natural, $\aB$-bilinear, and unitary maps $m_M(\sigma,\tau) : \Gamma_M(\sigma) \tensor_\aB \Gamma_M(\tau) \to \Gamma_M(\sigma \tensor \tau)$ for all $\sigma,\tau \in\mathcal{T}$.
To this end, we consider the canonical multiplication~maps $m_M(k,l) : M^{\tensor k} \tensor_\aB M^{\tensor l} \to M^{\tensor (k+l)}$ for all integers $k,l \geq 0$, which are obviously $\aB$-bilinear and unitary, and note that
\begin{align}\label{eq:compmult}
	m_M(k,l+m) \left( \id \tensor_\aB \, m_M(l,m) \right) = m_M(k+l,m) \left( m_M(k,l) \tensor_\aB \id \right)
\end{align}
for all integers $k,l,m \geq 0$.
Furthermore, we impose another condition ensuring that the maps $\varphi_{k,l}$ for all integers $k,l \geq 0$ are compatible with respect to taking tensor products:

\begin{defn}\label{def:pi-corr_tensor}
	Let $\aB$ be a unital C\Star-algebra and let $\pi$ be the standard representation of $\SO(n)$, $n \geq 3$.
	We say that a \emph{correspondence $M$ over $\aB$ is tensorial of type~$\pi$} if it is of type~$\pi$ (\cf~Definition~\ref{def:pi-corr}) and the following compatibility condition is satisfied:
	\begin{enumerate}[nosep]
	\item[(T)]
		$m_M(k,l) \big( \varphi_{k,k}(T) \tensor_\aB \varphi_{l,l}(T') \big) = \varphi_{k+l}(T \tensor T') m_M(k,l)$ for all $k,l \geq 0$, $T \in C_{k,k}$, and $T' \in C_{l,l}$.
	\end{enumerate}
\end{defn}

\begin{remark}\label{rem:free=pi-corr_tensor}
	Clearly, each free C\Star-dynamical system $(\aA,G,\alpha)$ with fixed point algebra $\aB$ gives rise to a correspondence over $\aB$ that is tensorial of type $\pi$ by looking at $\Gamma_\aA(\pi)$.
\end{remark}

In the remainder of this section we assume that $M$ is tensorial of type~$\pi$.
Let $(\sigma,V_\sigma)$ and $(\tau,V_\tau)$ be objects in $\mathcal{S}$, let $k,l \geq 0$ be integers such that $(\sigma,V_\sigma)$ and $(\tau,V_\tau)$ are subrepresentations of $\big(\pi^{\otimes k},V^{\otimes k}\big)$ and  $\big(\pi^{\otimes l},V^{\otimes l}\big)$, respectively, and let $P_\sigma$ and $P_\tau$ be the orthogonal projections of $V^{\otimes k}$ onto~$V_\sigma$ and $V^{\otimes l}$ onto~$V_\tau$, respectively.
Then $(\sigma \tensor \tau,V_\sigma \tensor V_\tau)$ is a subrepresentation of $\big(\pi^{\otimes (k+l)},V^{\otimes (k+l)}\big)$ and $P_\sigma \tensor P_\tau$ is the orthogonal~projection of $V^{\otimes (k+l)}$ onto~$V_\sigma \tensor V_\tau$.
We put
\begin{gather}
	m_M(\sigma,\tau): \Gamma_M(\sigma) \otimes_\aB \Gamma_M(\tau) \to \Gamma_M(\sigma \tensor \tau) \notag
	\\
	m_M(\sigma,\tau) := \underbrace{\varphi_{k+l}(P_\sigma \tensor P_\tau)}_{=\Gamma_M(P_\sigma \tensor P_\tau)} m_M(k,l) \restriction_{\Gamma_M(\sigma) \otimes_\aB \Gamma_M(\tau)}^{\Gamma_M(\sigma \tensor \tau)} 
	\label{eq:m(sigma,tau)}
\end{gather}
It is readily seen that the map $m_M(\sigma,\tau)$ is well-defined and $\aB$-bilinear.
Furthermore, Condition~\hyperref[def:pi-corr_tensor]{(T)} shows that 
\begin{align}\label{eq:m(sigma,tau)_alt}
	m_M(\sigma,\tau) = m_M(k,l) \big(  \underbrace{\varphi_{k,k}(P_\sigma) \tensor_\aB \varphi_{l,l}(P_\tau)}_{=\Gamma_M(P_\sigma) \otimes_\aB \Gamma_M(P_\tau)} \big) \restriction_{\Gamma_M(\sigma) \otimes_\aB \Gamma_M(\tau)}^{\Gamma_M(\sigma \tensor \tau)} 
\end{align}
We proceed with a series of~lemmas.

\begin{lemma}\label{lem:natural}
	We have $m_M(\sigma,\tau) \left( \Gamma_M(T) \otimes_\aB \Gamma_M(T') \right) = \Gamma_M(T \tensor T') m_M(\sigma,\tau)$ for all morphisms $T:V_\sigma \to V_\sigma$ and $T':V_\tau \to V_\tau$.
\end{lemma}
\begin{proof}
	Let $T:V_\sigma \to V_\sigma$ and $T':V_\tau \to V_\tau$ be morphisms.
	Using~Conditions~\hyperref[def:pi-corr_tensor]{(T)} and~\hyperref[def:pi-corr]{(C)}, on the domain $\Gamma_M(\sigma) \tensor_\aB \Gamma_M(\tau)$ we deduce that
	\begingroup
	\allowdisplaybreaks
	\begin{align*}
		m_M(\sigma,\tau) \left( \Gamma_M(T) \otimes_\aB \Gamma_M(T') \right)
		&=
		\varphi_{k+l}(P_\sigma \tensor P_\tau) m_M(k,l) \left( \varphi_{k,k}(W_T) \otimes_\aB \varphi_{l,l}(W_{T'}) \right)
		\\
		&=
		\varphi_{k+l}(P_\sigma \tensor P_\tau) \varphi_{k+l}(W_{T \tensor T'}) m_M(k,l)
		\\
		&=
		\varphi_{k+l}((P_\sigma \tensor P_\tau) W_{T \tensor T'}) m_M(k,l)
		\\
		&=
		\varphi_{k+l}(W_{T \tensor T'} (P_\sigma \tensor P_\tau)) m_M(k,l)
		\\
		&=
		\varphi_{k+l}(W_{T \tensor T'}) \varphi_{k+l}(P_\sigma \tensor P_\tau) m_M(k,l)
		\\
		&=
		\varphi_{k+l}(W_{T \tensor T'}) m_M(\sigma,\tau)
		=
		\Gamma_M(T \tensor T') m_M(\sigma,\tau).
		\qedhere
	\end{align*}
	\endgroup
\end{proof}

\begin{lemma}\label{lem:isometric}
	The map $m_M(\sigma,\tau)$ is isometric.
\end{lemma}
\begin{proof}
	Let $x \in \Gamma_M(\sigma)$ and let $y \in \Gamma_M(\tau)$.
	By~\eqref{eq:m(sigma,tau)_alt}, we have
	\begin{align*}
		\rprod{\aB}{m_M(\sigma,\tau)(x \otimes_\aB y),m_M(\sigma,\tau)(x \otimes_\aB y)} 
		&= 
		\rprod{\aB}{m_M(k,l)(x \otimes_\aB y),m_M(k,l)(x \otimes_\aB y)}.
	\end{align*}
		The claim therefore follows from the fact that the map $m_M(k,l)$ is unitary.
\end{proof}

\begin{lemma}\label{lem:surjective}
	The map $m_M(\sigma,\tau)$ is surjective.
\end{lemma}
\begin{proof}
	Our proof starts with the observation that 
	\begin{align*}
		M^{\otimes k} \otimes_\aB M^{\otimes l} = \Gamma_M(\sigma) \tensor_\aB \Gamma_M(\tau) \oplus \ker\left( \varphi_{k,k}(P_\sigma) \tensor_\aB \varphi_{l,l}(P_\tau) \right).
	\end{align*}
	Furthermore, we have $\ker\left( \varphi_{k+l}(P_\sigma \tensor P_\tau) m_M(k,l) \right) = \ker\left( \varphi_{k,k}(P_\sigma) \tensor_\aB \varphi_{l,l}(P_\tau) \right)$ which is clear from Condition~\hyperref[def:pi-corr_tensor]{(T)}.
	Since $\ran\left( \varphi_{k+l}(P_\sigma \tensor P_\tau) m_M(k,l) \right) = \Gamma_M(\sigma \tensor \tau)$,  it follows that $\ran(m_M(\sigma,\tau)) =\Gamma_M(\sigma \tensor \tau)$ as required.
\end{proof}

To summarize:

\begin{corollary}\label{cor:unitary_S}
	The map $m_M(\sigma,\tau) : \Gamma_M(\sigma) \tensor_\aB \Gamma_M(\tau) \to \Gamma_M(\sigma \tensor \tau)$ given by~\eqref{eq:m(sigma,tau)} is natural, $\aB$-bilinear, and unitary for all objects $\sigma,\tau$ in $\mathcal{S}$.
\end{corollary}

Note that $m_M(\pi^{\otimes k},\pi^{\otimes l}) = m_M(k,l)$ for all integers $k,l \geq 0$.
Our next claim is that the maps $m_M(\sigma,\tau)$ for all objects $\sigma,\tau$ in $\mathcal{S}$ satisfy the following associativity condition:

\begin{lemma}\label{lem:asso}
	We have 
	\begin{align*}
		m_M(\sigma, \tau \tensor \rho) \left( \id \tensor_\aB \, m_M(\tau, \rho) \right) 
		=
		m_M(\sigma \tensor \tau, \rho) \left( m_M(\sigma, \tau) \tensor_\aB \id \right)
	\end{align*}
	for all objects $\sigma, \tau, \rho$ in $\mathcal{S}$.
\end{lemma}
\begin{proof}
	Let $(\sigma,V_\sigma)$, $(\tau,V_\tau)$, and $(\rho,V_\rho)$ be objects in $\mathcal{S}$.
	Furthermore, let $k$,  $l$ and $m$ be non-negative integers such that $(\sigma,V_\sigma)$, $(\tau,V_\tau)$, and $(\rho,V_\rho)$ are subrepresentations of $\big(\pi^{\otimes k},V^{\otimes k}\big)$, $\big(\pi^{\otimes l},V^{\otimes l}\big)$, and $\big(\pi^{\otimes m},V^{\otimes m}\big)$, respectively, and let $P_\sigma,P_\tau$, and $P_\rho$ be the orthogonal projections of $V^{\otimes k}$ onto~$V_\sigma$, $V^{\otimes l}$ onto~$V_\tau$, and $V^{\otimes m}$ onto~$V_\rho$, respectively.
	On the domain $\Gamma_M(\sigma) \tensor_\aB \Gamma_M(\tau) \tensor_\aB \Gamma_M(\rho)$ we find that
	\begingroup
	\allowdisplaybreaks
	\begin{align*}
		&m_M(\sigma, \tau \tensor \rho) \left( \id \tensor_\aB \, m_M(\tau, \rho) \right)
		\\
		\overset{\eqref{eq:m(sigma,tau)}}&{=}
		\varphi_{k+l+m}(P_\sigma \tensor (P_\tau \tensor P_\rho)) m_M(k,l+m) 
		\left( \varphi_{k,k}(P_\sigma) \tensor_\aB \left( \varphi_{l+m}(P_\tau \tensor P_\rho) m_M(l,m) \right) \right)
		\\
		\overset{\hyperref[def:pi-corr_tensor]{(\text{T})}}&{=}
		\varphi_{k+l+m}(P_\sigma \tensor P_\tau \tensor P_\rho) 
		\varphi_{k+l+m}(P_\sigma \tensor P_\tau \tensor P_\rho)
		m_M(k,l+m) \left( \id \tensor_\aB \, m_M(l,m) \right)
		\\
		\overset{\eqref{eq:compmult}}&{=}
		\varphi_{k+l+m}(P_\sigma \tensor P_\tau \tensor P_\rho)
		\varphi_{k+l+m}(P_\sigma \tensor P_\tau \tensor P_\rho)
		m_M(k+l,m) \left( m_M(k,l) \tensor_\aB \id \right)
		\\
		\overset{\hyperref[def:pi-corr_tensor]{(\text{T})}}&{=}
		\varphi_{k+l+m}((P_\sigma \tensor P_\tau) \tensor P_\rho) m_M(k+l,m) 
		\left( \varphi_{k+l}(P_\sigma \tensor P_\tau) m_M(k,l) \tensor_\aB \,\varphi_{m,m}(P_\rho) \right).
		\\
		\overset{\eqref{eq:m(sigma,tau)}}&{=}
		m_M(\sigma \tensor \tau, \rho) \left( m_M(\sigma, \tau) \tensor_\aB \id \right).
		\qedhere
	\end{align*}
	\endgroup
\end{proof}

Finally, we look once more at the small tensor category $\mathcal{T} \subseteq \rep(\SO(n))$.
For two objects $\sigma = \sigma_1 \oplus \dots \oplus \sigma_r$ and $\tau = \tau_1 \oplus \dots \oplus \tau_s$ in $\mathcal{T}$ we put
\begin{align*}
	m_M(\sigma,\tau) := \bigoplus_{i,j} m_M(\sigma_i,\tau_j).
\end{align*}
A straightforward verification shows that this formula yields a natural $\aB$-bilinear unitary map $m_M(\sigma,\tau) : \Gamma_M(\sigma) \tensor_\aB \Gamma_M(\tau) \to \Gamma_M(\sigma \tensor \tau)$ for all objects $\sigma,\tau$ in $\mathcal{T}$ satisfying 
\begin{align*}
	m_M(\sigma, \tau \tensor \rho) \left(\id \tensor_\aB \, m_M(\tau, \rho)\right) = m_M(\sigma \tensor \tau, \rho) \left(m_M(\sigma, \tau) \tensor_\aB \id\right)
\end{align*}
for all objects $\sigma, \tau, \rho$ in $\mathcal{T}$.
We can now present the main results of this paper:

\begin{theorem}\label{thm:unitenfun}
	The linear functor $\Gamma_M: \mathcal{T} \to \corr(\aB)$ together with the natural $\aB$-bilinear unitary maps $m_M(\sigma,\tau) : \Gamma_M(\sigma) \tensor_\aB \Gamma_M(\tau) \to \Gamma_M(\sigma \tensor \tau)$ for all $\sigma,\tau \in \mathcal{T}$ constitute a unitary tensor functor such that $\Gamma_M\big(\pi^{\otimes k}\big) = M^{\tensor k}$ for all integers $k \geq 0$.
\end{theorem}

Having the unitary tensor functor $\mathcal{T} \to \corr(\aB)$ at our disposal, we can now apply the construction procedure presented in~\cite[Thm.~2.3]{Ne13} (see also~\cite[Thm.~5.6]{SchWa17}) in order to obtain a free C\Star-dynamical system $(\aA_M,\SO(n),\alpha_M)$ with fixed point algebra $\aB$ and $\Gamma_{\aA_M}(\pi) \cong M$.
This is the desired \emph{noncommutative frame bundle associated with~$M$}.
For the sake of expediency we briefly sketch the main steps of the construction of $(\aA_M,\SO(n),\alpha_M)$ in Appendix~\ref{sec:construction}.

\begin{corollary}\label{cor:main1}
	Let $\aB$ be a unital C\Star-algebra and let $\pi$ be the standard representation of $\SO(n)$, $n \geq 3$.
	Each correspondence $M$ over~$\aB$ that is tensorial of type~$\pi$ yields a unitary tensor functor $\mathcal{T} \to \corr(\aB)$ for some small tensor subcategory $\mathcal{T}$ of $\rep(SO(n))$ containing representatives of $\Irrep(\SO(n))$ such that $\Gamma_M\big(\pi^{\otimes k}\big) = M^{\tensor k}$ for all integers $k \geq 0$, and hence a free C\Star-dynamical system $(\aA_M,\SO(n),\alpha_M)$ with fixed point algebra $\aB$ and $\Gamma_{\aA_M}(\pi) \cong \Gamma_M(\pi) = M$.
\end{corollary}

\begin{remark}
	By the general theory of free C\Star-dynamical systems, it follows that $M$, and hence $\Gamma_M(\sigma)$ for all $\sigma \in \mathcal{T}$, is necessarily finitely generated and projective as a right $\aB$-module (see,~\eg,~\cite[Thm.~1.2]{CoYa13a}).
\end{remark}

We conclude this section with a classification result that extends the classical correspondence between frame bundles and their associated vector bundles.
For this, we consider equivalence classes of free C\Star-dynamical systems with structure group $\SO(n)$ and fixed point algebra $\aB$ with respect to equivariant \Star-isomorphisms that preserve $\aB$ and, further, equivalence classes of correspondences over $\aB$ with respect to $\aB$-bilinear isomorphisms.
Combining Corollary~\ref{cor:main1} with~\cite[Thm.~2.3]{Ne13}, we get:

\begin{corollary}\label{cor:main2}
	Let $\aB$ be a unital C\Star-algebra and let $\pi$ be the standard representation of $\SO(n)$, $n \geq 3$.
	The map $[(\aA,\SO(n),\alpha)] \mapsto [\Gamma_\aA(\pi)]$ yields a bijective correspondence between equivalence classes of free C\Star-dynamical systems with structure group $\SO(n)$ and fixed point algebra $\aB$ and equivalence classes of correspondences over~$\aB$ that are tensorial of type~$\pi$ with inverse given by $[M] \mapsto [(\aA_M,\SO(n),\alpha_M)]$.
\end{corollary}

\pagebreak[3]
\section{Examples}\label{sex:examples}

This section is devoted to discussing examples.
For expediency we continue to write~$(\pi,V)$ for the standard representation of $\SO(n)$, $n \geq 3$.


\enlargethispage{\baselineskip}
\subsection{\boldmath Example: the free module of rank~$n$}\label{ex:trivial}

Let $\aB$ be a unital C\Star-algebra.
In this example we apply ourselves to the tensor product $M := \aB \tensor V$ which is naturally a correspondence over $\aB$ with respect to the canonical left and right actions and the right $\aB$-valued inner product determined by $\rprod{\aB}{b \tensor v, b' \tensor v'} := \scal{v,v'} b^*b'$ for all $b,b' \in \aB$ and $v,v' \in V$. 
Note that, up to the canonical isomorphism $\aB^{\tensor k} \cong \aB$, we have
$M^{\tensor k} = \aB \tensor V^{\tensor k}$ for all integers $k \geq 0$.
For integers $k,l \geq 0$, $T \in C_{k,l}$, and $x \in M^{\tensor k}$ we obtain an element in $M^{\tensor l}$ by putting $\varphi_{k,l}(x) := \id_\aB \tensor T(x)$.
This yields injective linear maps $\varphi_{k,l} : C_{k,l} \to \End\big(M^{\tensor k},M^{\tensor l}\big)$ for all integers $k,l \geq 0$ which make~$M$ tensorial of type~$\pi$, as is easy to check.
From the construction procedure presented in Appendix~\ref{sec:construction} we infer that the corresponding free C\Star-dynamical system $(\aA_M,\SO(n),\alpha_M)$ is equivalent to $(\aB \tensor_{\text{min}} C(\SO(n)),\SO(n),\alpha)$, where $\tensor_{\text{min}}$ denotes the minimal tensor product of C\Star-algebras and the action $\alpha$ is given by right translation in the argument of the second tensor factor.

\subsection{Example: classical vector bundles}\label{ex:VB}

Let $X$ be a compact space and let $q:E \to X$ be a locally trivial hermitian vector bundle with typical fibre $V$, structure group $\SO(n)$, and Hermitian metric $x \mapsto \scal{\cdot,\cdot}_x$. 
In this example we consider the space $M := \Gamma(E)$ of continuous sections of $q:E \to X$ which carries the structure of a correspondence over $C(X)$ with respect to the obvious (bi-)module structure given by pointwise multiplication and the inner product $\rprod{C(X)}{\cdot,\cdot}$ given for $s,t \in \Gamma(E)$ by $\rprod{C(X)}{s,t}(x) := \scal{s(x),t(x)}_x$, $x \in X$.
In case $E$ is trivial, $M \cong C(X) \tensor V$, and therefore we are, up to isomorphism, in the situation of Example~\ref{ex:trivial}.
In particular, we have $\aA_M \cong C(X) \tensor_{\text{min}} C(\SO(n))$ which shows that $\aA_M$ is commutative with character space given by the trivial frame bundle $\Fr(E) \cong X \times \SO(n)$.
In case $E$ is non-trivial, we use bundle charts to conclude similarly that $M$ is tensorial of type~$\pi$.
Let $(\aA_M,\SO(n),\alpha_M)$ be the corresponding free C\Star-dynamical system.
A moment's thought shows that $m(\sigma,\tau) = \textit{flip}(m(\tau,\sigma))$ for all $\sigma,\tau \in \mathcal{T}$, where $\textit{flip}$ denotes the fibrewise tensor flip.
From this it follows that $\aA_M$ is commutative, and hence that $\aA_M \cong C(\Fr(E))$ by the uniqueness (up to isomorphism) of the geometric construction.

\begin{remark}
	If $X$ is a closed orientable manifold, then its tangent space is a locally trivial hermitian vector bundle with typical fibre $V$ and structure group $\SO(n)$.
\end{remark}

\subsection{Example: the quantum projective 7-space}\label{ex:quantumprojective}

In this example we introduce a new free C\Star-dynamical system with structure group $\SO(3)$.
In particular, its associated noncommutative vector bundle with respect to the standard representation $\pi$ of $\SO(3)$ yields another instance of a tensorial correspondence of type $\pi$ (\cf~Remark~\ref{rem:free=pi-corr_tensor}).
To the best of our knowledge, this noncommutative principal bundle has, as of yet, not been considered in the literature.

For a start we recall a noncommutative C\Star-algebraic version of the classical $\SU(2)$-Hopf fibration over the four sphere (see~\cite{LaSu05} for a generalization in the context of Hopf-Galois extensions).
Let $\theta \in \R$ and let $\theta'$ be the skewsymmetric $4 \times 4$-matrix with $\theta_{1,2}' = \theta_{3,4}' = 0$ and $\theta'_{1,3} = \theta'_{1,4} = \theta_{2,3}' = \theta'_{2,4} = \theta/2$. 
The Connes-Landi sphere $\aA(\mathbb S_{\theta'}^7)$ is the universal unital C\Star-algebra generated by normal elements $z_1, \dots, z_4$ subject to the relations
\begin{align*}	
	z_i z_j &= e^{2\pi\imath \theta'_{i,j}} \; z_j z_i, 
	&
	z_j^* z_i &= e^{2\pi \imath \theta'_{i,j}}\;  z_i z_j^*,
	&
	\sum_{k=1}^4 z_k^* z_k^{} &= \one
\end{align*}
for all $1 \le i,j \le 4$. 
By~\cite[Expl.~3.5]{SchWa17}, it comes equipped with a free action $\alpha$ of $\SU(2)$ given for each $U \in \SU(2)$ on generators by
\begin{equation*}
	\alpha_U: (z_1, \dots, z_4) \mapsto (z_1, \dots, z_4) \begin{pmatrix} U & 0 \\ 0 & U \end{pmatrix}.
\end{equation*}
The corresponding fixed point algebra is the universal unital C\Star-algebra $\aA(\mathbb S_\theta^4)$ generated by normal elements $w_1, w_2$ and a self-adjoint element $x$ satisfying
\begin{align*}
	w_1 w_2 &= e^{2\pi\imath\theta} \; w_2 w_1, 
	&
	w_2^* w_1 &= e^{2\pi\imath \theta} \; w_1 w_2^*, 
	&
	w_1^* w_1 + w_2^* w_2 + x^*x &= \one.
\end{align*}
To proceed, we consider the normal subgroup $N := \{\pm \one\} \subseteq \SU(2)$.
\cite[Prop.~3.18]{SchWa15}~implies that the induced C\Star-dynamical system $(\aA(\mathbb S_{\theta'}^7)^N,\SU(2)/N,\alpha\restriction_{\SU(2)/N})$ is free, too.
Hence we arrive at the announced free C\Star-dynamical system with structure group $\SO(3)$ by (simply) putting $\aA(\mathbb P_{\theta'}^7) := \aA(\mathbb S_{\theta'}^7)^N$ and by identifying $\SO(3)$ with $\SU(2)/N$ via the universal covering map $p:\SU(2) \to \SO(3)$, \ie, 
\begin{align*}
	(\aA(\mathbb P_{\theta'}^7),\SO(3),\alpha\restriction_{\SU(2)/N} \circ \, {\bar p}^{-1}),
\end{align*}
$\bar p$ being the induced isomorphism $\SU(2)/N \to \SO(3)$.
Note that $\aA(\mathbb P_{\theta'}^7)^{\SO(3)} = \aA(\mathbb S_\theta^4)$.
Finally, we conclude from Remark~\ref{rem:free=pi-corr_tensor} that $\Gamma_{\aA(\mathbb P_{\theta'}^7)}(\pi)$ is tensorial of type $\pi$.


\subsection{\boldmath Example: the even part of the Cuntz algebra $\mathcal{O}_2$}\label{ex:cuntz}

In this example we present yet another instance of a free C\Star-dynamical system with structure group~$\SO(3)$, and hence of a correspondence that is tensorial of type $\pi$.
Apparently, this free C\Star-dynamical system has neither been considered elsewhere in the literature.

To begin with, we recall that the Cuntz algebra $\mathcal{O}_2$ is the universal unital C\Star-algebra generated by two elements $S_1$ and $S_2$ satisfying $S_i^* S_j = \delta_{ij}$ and $S_1S_1^*+S_2S_2^* = \one$ (\cf~\cite{Cuntz77}).
On account of~\cite[Prop.~8.4]{Gab14}, it comes equipped with a free action $\alpha$ of $\SU(2)$ given for each $U := \begin{psmallmatrix*}[r] a & -\bar b\\ b & \bar a \end{psmallmatrix*} \in \SU(2)$ on generators by
\begin{align*}
	\alpha_U: (S_1,S_2) \mapsto (aS_1 + bS_2,-\bar b S_1 + \bar a S_2).
\end{align*}
Now, the exact same line of arguments as in Example~\ref{ex:quantumprojective} above shows that the induced C\Star-dynamical system $(\mathcal{O}_2^N,\SO(3),\alpha\restriction_{\SU(2)/N} \circ \, {\bar p}^{-1})$ is free and, further, that $\Gamma_{\mathcal{O}_2^N}(\pi)$ is tensorial of type $\pi$.

\appendix

\pagebreak[3]
\section{The construction of the free \texorpdfstring{C$^*$}{C*}-dynamical systems from the unitary tensor functor}\label{sec:construction}

This section contains a brief summary of the construction of the free C\Star-dynamical system $(\aA_M,\SO(n),\alpha_M)$ from the unitary tensor functor $\mathcal{T} \to \corr(\aB)$ put together in~Section~\ref{sec:NCFB}. 
Some parts of the construction require us to deal with conjugates.
Indeed, given an irreducible representation $(\sigma,V_\sigma)$ of $\SO(n)$, we identify its conjugate representation $(\bar \sigma,\bar V_\sigma)$ with the equivalent irreducible representation from our initial choice of representatives and denote the latter, by abuse of notation, also by $(\bar \sigma,\bar V_\sigma)$.

First, we form an algebra $A_M$.
To do this, we consider the algebraic direct sum
\begin{align*}
	A_M := \bigoplus_{\sigma \in \Irrep(\SO(n))} \Gamma_M(\bar \sigma) \tensor V_\sigma.
\end{align*} 
Moreover, for $\sigma,\tau \in \Irrep(\SO(n))$, $x \tensor v \in \Gamma_M(\bar \sigma) \tensor V_\sigma$, and $y \tensor w \in \Gamma_M(\bar \tau) \tensor V_\tau$ we define a product by the recipe
\begin{equation*}
	(x \tensor v) \bullet (y \tensor w) := \sum_{k=1}^N \left( \Gamma_M\left(\bar S_k\right)^* \tensor S_k^* \right) \left( m(\bar \sigma,\bar \tau)(x \tensor y) \tensor v \tensor w \right)
	\in \sum_{k=1}^N \Gamma_M\left(\bar \sigma_k\right) \tensor V_{\sigma_k},
\end{equation*}
where $\{S_1, \dots, S_N\}$ is a complete set of isometric intertwiners $S_k:V_{\sigma_k} \to V_\sigma \tensor V_\tau$, $\sigma_k \in \Irrep(\SO(n))$, with respective conjugates $\bar S_k:\bar V_{\sigma_k} \to \bar V_\sigma \tensor \bar V_\tau$. 
Extending this product bilinearly yields a multiplication on $A_M$ which is associative due to condition~(iii) in the definition of a~\hyperref[defn:wutfun]{unitary tensor functor}.
Note that $\aB$ can be regarded as the subalgebra of~$A_M$ corresponding to the equivalence class of the trivial representation.
Second, we turn $A_M$ into a \Star-algebra.
For this purpose, let $\sigma \in \Irrep(\SO(n))$.
We define an involutive map $^+ : \Gamma_M(\sigma) \to \Gamma_M(\bar \sigma)$ by setting $x^+ := m_x^* \left( \Gamma_M(R)(\one_\aB) \right)$, where $R:\C \to V_\sigma \tensor \bar V_\sigma$ is any intertwiner and $m_x:\Gamma_M(\bar \sigma) \to \Gamma_M(\sigma \tensor \bar \sigma)$ denotes the map $m_x(y) := m(\sigma,\bar \sigma)(x \tensor y)$.
Now, for $x \tensor v \in \Gamma_M(\bar \sigma) \tensor V_\sigma$ we put $(x \tensor v)^+ := x^+ \tensor \bar v$ and extend this anilinearly to an involutive map on $A_M$.
Third, we equip $A_M$ with the $\SO(n)$-action by \Star-automorphisms, let's say $a_M$, given on each summand $\Gamma_M(\bar \sigma) \tensor V_\sigma$, $\sigma \in \Irrep(\SO(n))$, by the respective unitary representation of $\SO(n)$ on the second tensor factor. 
In summary, we have built a \Star-algebra $A_M$ together with an action of $\SO(n)$ on $A_M$ by \Star-automorphisms.

We proceed by noting that each summand $\Gamma_M(\bar \sigma) \tensor V_\sigma$, $\sigma \in \Irrep(\SO(n))$, is naturally a correspondence over $\aB$ with respect to the canonical $\aB$-bimodule structure and the $\aB$-valued inner product determined by $\rprod{\aB}{x \tensor v, y \tensor w} := \scal{x,y}_\aB \scal{v,w}$ for all $x,y \in \Gamma_M(\bar \sigma)$ and $v,w \in V_\sigma$.
From this it follows that $A_M$ carries the structure of a right pre-Hilbert $\aB$-module.
We write $\hH_M$ for its completion.
It is easy to check that the left multiplication on $A_M$ yields a faithful \Star-representation $\lambda: A_M \to \End(\hH_M)$.
Furthermore, it is immediate that $a_M$ extends to a unitary representation $U_M : \SO(n) \to \mathcal{U}(\hH_M)$.

Now, we are in a position to introduce the free C\Star-dynamical system $(\aA_M,\SO(n),\alpha_M)$.
Indeed, we let $\aA_M$ be the closure of $\lambda(A_M)$ with respect to the operator norm on $\End(\hH_M)$.
Furthermore, $\alpha_M$ is implemented by the unitary representation $U_M$ in the sense that $(\alpha_M)_g(x) = (U_M)_g x (U_M)_g^*$ for all $g \in \SO(n)$ and $x \in \aA_M$.
Finally, $(\aA_M,\SO(n),\alpha_M)$ is free as asserted, because we initially started with a unitary tensor functor.
For a more detailed account of the construction we refer to~\cite[Sec.~2]{Ne13}.

\pagebreak[3]
\section{\boldmath The special case of $\SO(2)$}\label{sec:MC}

In this section we briefly deal with the special case of $\SO(2)$.
More precisely, for a unital C\Star-algebra $\aB$ we show that there is a bijective correspondence between free C\Star-dynamical systems with structure group $\SO(2)$ and fixed point algebra $\aB$ and Morita equivalence $\aB$-bimodules.
We begin by recalling that $\Irrep(\SO(2) \cong \mathbb{Z}$.
Now, let $(\aA,\SO(2),\alpha)$ be a free C\Star-dynamical system with fixed point algebra $\aB$.
Each isotypic component $A(k)$, $k \in \Z$, is a  Morita equivalence $\aB$-bimodule with inner products $\lprod{\aB}{x,y} := xy^*$ and $\rprod{\aB}{x,y} := x^*y$ for all $x,y \in A(k)$.
Furthermore, the canonical multiplication maps
\begin{align*}
	\Psi_{k_1,k_2}: A(k_1) \otimes_\aB  A(k_2) \to A(k_1+k_2),
	&&
	x \otimes_\aB y \mapsto xy
\end{align*}
are isomorphisms of Morita equivalence $\aB$-bimodules for all $k_1,k_2 \in \Z$ and the following associativity condition holds for all $k_1,k_2,k_3 \in \Z$:
\begin{align}\label{eq:asso}
	\Psi_{k_1+k_2,k_3} (\Psi_{k_1,k_2} \otimes_\aB \id_{k_3}) = \Psi_{k_1,k_2+k_3} (\id_{k_1} \otimes_\aB \Psi_{k_2,k_3}).
\end{align}
Note that $A(0) = \aB$ and that, up to the canonical isomorphism $\aB \otimes_\aB \aB \cong \aB$, $\Psi_{0,0} = \id_\aB$.
The isotypic components $A(k)$, $k \in \mathbb{Z}$, along with the maps $\Psi_{k_1,k_2}$, $k_1,k_2 \in \mathbb{Z}$, constitute a so-called factor system and allows to reconstruct the C\Star-dynamical system $(\aA,\SO(2),\alpha)$ up to isomorphism (\cf~\cite[Def.~4.5]{SchWa15}).
Conversely, let $\aB$ be a unital C\Star-algebra and let $N$ be a Morita equivalence $\aB$-bimodule with dual module $\bar{N}$ and isomorphisms
\begin{align*}
	\Psi_{1,-1} : N \otimes_\aB \bar{N} \to \aB, && &x \otimes \tilde{y} \mapsto \lprod{\aB}{x,y},
	\shortintertext{and}
	\Psi_{-1,1} : \bar{N} \otimes_\aB N \to \aB, && &\tilde{x} \otimes y \mapsto \rprod{\aB}{x,y}.
\end{align*}
Note that the Morita equivalence condition $\lprod{\aB}{x,y} \acts z = x \acts \rprod{\aB}{y,z}$ for all $x,y,z \in N$ implies that $\Psi_{1,-1} \otimes_\aB \id_N = \id_N \otimes_\aB \Psi_{-1,1}$ and $\Psi_{-1,1} \otimes_\aB \id_{\bar{N}} = \id_{\bar{N}} \otimes_\aB \Psi_{1,-1}$.
The task is now to construct a free C\Star-dynamical system $(\aA_N,\SO(2),\alpha_N)$ with fixed point algebra $\aB$ and $A_N(1) = N$.
For this, we apply~\cite[Thm.~4.21]{SchWa15} which amounts to the construction of a factor system associated with $N$:
First, for each $k \in \Z$ we form a Morita equivalence $\aB$-bimodule $N(k)$ by setting
\begin{align*}
	N(k):=
	\begin{cases}
		\aB & \quad k = 0
		\\
		N^{\otimes k} & \quad k > 0
		\\
		\bar{N}^{\otimes -k} & \quad k < 0.
	\end{cases}
\end{align*}
Second, we define Morita equivalence $\aB$-bimodule isomorphisms
\begin{align*}
	\Psi_{k_1,k_2}: N(k_1) \otimes_\aB  N(k_2) \to N(k_1+k_2)
\end{align*}
for all $k_1,k_2 \in \Z$ in the following way:
For non-negative integers we simply take the tensor product $\otimes_\aB$,  and similarly for negative integers.
Note that $N(0) = \aB$ and that, up to the canonical isomorphism $\aB \otimes_\aB \aB \cong \aB$, $\Psi_{0,0} = \id_\aB$.
To deal with integers of mixed parity, we repeatedly make use of the maps $\Psi_{1,-1}$ and $\Psi_{-1,1}$. 
From $\Psi_{1,-1} \otimes_\aB \id_1 = \id_1 \otimes_\aB \Psi_{-1,1}$ and $\Psi_{-1,1} \otimes_\aB \id_{-1} = \id_{-1} \otimes_\aB \Psi_{1,-1}$, it may be concluded that Equation~\eqref{eq:asso} holds for all $k_1,k_2,k_3 \in \Z$.
Hence the modules $N(k)$, $k \in \mathbb{Z}$, together with the maps $\Psi_{k_1,k_2}$, $k_1,k_2 \in \mathbb{Z}$, form a factor system which gives a free C\Star-dynamical system $(\aA_N,\SO(2),\alpha_N)$ with fixed point algebra $\aB$ and $A(1) = N$ as required.
We proceed by looking at the maps
\begin{align*}
	(\aA,\SO(2),\alpha) \mapsto A(1), 
	&&
	\text{and}
	&&
	N \mapsto (\aA_N,\SO(2),\alpha_N).
\end{align*}
A few moments thought show that these are, in fact, inverse to each other as claimed in the beginning of this section.
Passing over to the set $\Ext(\aB,\SO(2))$ of equivalence classes of free C\Star-dynamical systems with structure group $\SO(2)$ and fixed point algebra $\aB$ (with respect to $\SO(2)$-equivariant isomorphisms over $\aB$) and the set $\Pic(\aB)$ of equivalence classes of Morita equivalence $\aB$-bimodules, the Picard group of $\aB$, we can assert that
\begin{align*}
	\Ext(\aB,\SO(2)) \to \Pic(\aB), 
	&&
	[(\aA,\SO(2),\alpha)] \mapsto [A(1)]. 
\end{align*}
is a bijection.

\begin{remark}
	The above correspondence can also be obtained from a more abstract result involving group cohomology. 
	In fact, since the group cohomology $\text{H}^n_{\text{gr}}(\mathbb{Z},\mathcal{U} Z(\aB))$ vanishes for each $n>1$, the correspondence follows from~\cite[Cor.~5.9 \& Thm.~5.14]{SchWa15}.
\end{remark}

\begin{remark}
	If $\aB= C(X)$ for some compact space $X$ and $N$ is the $C(X)$-module of sections of some line bundle $L$ over $X$, then $\aA_N \cong C(\Fr(L))$ (see, \eg,~\cite[Sec.~6]{SchWa17}).
\end{remark}

\begin{remark}
	For a similar discussion in the purely algebraic setting of strongly graded rings we refer to the opus~\cite[Sec.~5.2.3]{BeMa20}.
\end{remark}

We conclude by establishing a relation between free C\Star-dynamical systems with structure group $\SO(2)$ and their associated noncommutative vector bundles with respect to the standard representation of $\SO(2)$.
For this, we recall that the standard representation of $\SO(2)$ is not irreducible.
In fact, we have $\mathbb{C}^2 = \mathbb{C} (1,\imath)^\intercal \bigoplus \mathbb{C} (1,-\imath)^\intercal$ as $\SO(2)$-modules.

\begin{corollary}
	Let $(\aA,\SO(2),\alpha)$ be a free C\Star-dynamical system.
	Furthermore,  let $\pi$ be the standard representation of $\SO(2)$.
	Then $\Gamma_\aA(\pi) = A(1) \oplus A(-1)$.
	In particular, for a unital C\Star-algebra $\aB$ and a Morita equivalence $\aB$-bimodule $N$ with dual module $\bar{N}$ we have $\Gamma_{\aA_N}(\pi) = N \bigoplus \bar{N}$.
\end{corollary}

\enlargethispage{\baselineskip}
\section*{Acknowledgement}

The author wishes to express his thanks to Ludwik Dabrowski, Karl-Hermann Neeb, and Sergey Neshveyev, for helpful conversations and correspondence, and he especially thanks Kay Schwieger for numerous technical conversations over the last several years that have indelibly shaped this work.

\enlargethispage{\baselineskip}
\bibliographystyle{abbrv}
\bibliography{short,RS}

\begin{thebibliography}{10}

\bibitem{BaCoHa15}
P.~Baum, K.~D. Commer, and P.~M. Hajac.
\newblock Free actions of compact quantum groups on unital {C$^*$}-algebras.
\newblock {\em Doc. Math.}, 22:825--849, 2017.

\bibitem{BeMa20}
E.~Beggs and S.~Majid.
\newblock {\em Quantum Riemannian Geometry}.
\newblock Grundlehren der mathematischen Wissenschaften. Springer International
  Publishing, 2020.

\bibitem{BB06}
B.~Blackadar.
\newblock {\em Operator algebras. {Theory} of {C$^*$}-algebras and von
  {Neumann} algebras}, volume 122 of {\em Encycl. Math. Sci.}
\newblock Berlin: Springer, 2006.

\bibitem{BroDi85}
T.~Br\"{o}cker and T.~tom Dieck.
\newblock {\em Representations of compact {L}ie groups}, volume~98 of {\em
  Grad. Texts Math.}
\newblock Springer-Verlag, New York, 1985.

\bibitem{CoYa13a}
K.~D. Commer and M.~Yamashita.
\newblock A construction of finite index {C$^*$}-algebra inclusions from free
  actions of compact quantum groups.
\newblock {\em Publ. Res. Inst. Math. Sci.}, 49(4):709--735, 2013.

\bibitem{Cuntz77}
J.~Cuntz.
\newblock Simple {{\(C^*\)}}-algebras generated by isometries.
\newblock {\em Commun. Math. Phys.}, 57:173--185, 1977.

\bibitem{EKQR06}
S.~Echterhoff, S.~Kaliszewski, J.~Quigg, and I.~Raeburn.
\newblock A categorical approach to imprimitivity theorems for
  {C$^*$}-dynamical systems.
\newblock {\em Mem. Amer. Math. Soc.}, 180(850):viii+169, 2006.

\bibitem{Ell00}
D.~A. Ellwood.
\newblock A new characterisation of principal actions.
\newblock {\em J. Func. Anal.}, 173(1):49--60, 2000.

\bibitem{EtGENiOs15}
P.~Etingof, S.~Gelaki, D.~Nikshych, and V.~Ostrik.
\newblock {\em Tensor categories}, volume 205 of {\em Math. Surv. Monogr.}
\newblock Providence, RI: American Mathematical Society (AMS), 2015.

\bibitem{Gab14}
O.~Gabriel.
\newblock Fixed points of compact quantum groups actions on cuntz algebras.
\newblock {\em Annales Henri Poincar{\'e}}, 15(5):1013--1036, 2014.

\bibitem{GoWa09}
R.~Goodman and N.~R. Wallach.
\newblock {\em Symmetry, representations, and invariants}, volume 255 of {\em
  Grad. Texts Math.}
\newblock Springer-Verlag, New York, 2009.

\bibitem{HoMo06}
K.~H. Hofmann and S.~A. Morris.
\newblock {\em The structure of compact groups}, volume~25 of {\em De Gruyter
  Studies in Mathematics}.
\newblock De Gruyter, Berlin, 2013.

\bibitem{Huse94}
D.~Husem{\"o}ller.
\newblock {\em Fibre Bundles}, volume~20 of {\em Grad. Texts Math.}
\newblock Springer-Verlag, New York, 1994.

\bibitem{Kob69I}
S.~Kobayashi and K.~Nomizu.
\newblock {\em Foundations of differential geometry. {V}ol. {I}}.
\newblock Wiley Classics Library. John Wiley \& Sons, Inc., New York, 1996.

\bibitem{LaSu05}
G.~Landi and W.~v. Suijlekom.
\newblock Principal fibrations from noncommutative spheres.
\newblock {\em Comm. Math. Phys}, 260(1):203--225, 2005.

\bibitem{Maj99}
S.~Majid.
\newblock Quantum and braided group {Riemannian} geometry.
\newblock {\em J. Geom. Phys.}, 30(2):113--146, 1999.

\bibitem{Maj02}
S.~Majid.
\newblock Riemannian geometry of quantum groups and finite groups with
  nonuniversal differentials.
\newblock {\em Commun. Math. Phys.}, 225(1):131--170, 2002.

\bibitem{Maj05}
S.~Majid.
\newblock Noncommutative {Riemannian} and spin geometry of the standard
  {{\(q\)}}-sphere.
\newblock {\em Commun. Math. Phys.}, 256(2):255--285, 2005.

\bibitem{Mi65}
B.~Mitchell.
\newblock {\em Theory of categories}.
\newblock New {York} and {London}: {Academic} {Press}. {XI}, 273 p. (1965).,
  1965.

\bibitem{Ne13}
S.~Neshveyev.
\newblock Duality theory for nonergodic actions.
\newblock {\em M{\"u}nster J. Math.}, 7(2):414--437, 2013.

\bibitem{NeTu10}
S.~Neshveyev and L.~Tuset.
\newblock {\em Compact quantum groups and their representation categories},
  volume~20 of {\em Cours Sp\'ecialis\'es \emph{[}Specialized Courses\emph{]}}.
\newblock Soci\'et\'e Math\'ematique de France, Paris, 2013.

\bibitem{Ped18}
G.~K. Pedersen.
\newblock {\em {C$^*$}-algebras and their automorphism groups}.
\newblock Pure and Applied Mathematics (Amsterdam). Academic Press, London,
  2018.
\newblock Second edition, Edited and with a preface by S\o ren Eilers and Dorte
  Olesen.

\bibitem{Phi87}
C.~N. Phillips.
\newblock {\em Equivariant {K}-Theory and Freeness of Group Actions on
  {C$^*$}-Algebras}, volume 1274 of {\em Lecture Notes in Math.}
\newblock Springer, 1987.

\bibitem{Rae98}
I.~Raeburn and D.~Williams.
\newblock {\em Morita Equivalence and Continuous-Trace {C$^*$-algebras}},
  volume~60 of {\em Math. Surv. Monogr.}
\newblock Providence, RI: American Mathematical Society (AMS), 1998.

\bibitem{Rieffel91}
M.~A. Rieffel.
\newblock Proper actions of groups on {C$^*$}-algebras.
\newblock In {\em Mappings of Operator Algebras}, volume~84 of {\em Progr.
  Math.}, pages 141--182. Birkh{\"a}user, 1991.

\bibitem{SchWa15}
K.~Schwieger and S.~Wagner.
\newblock {Part} {I}, {Free} actions of compact {A}belian groups on
  {C$^*$}-algebras.
\newblock {\em Adv. Math.}, 317:224--266, 2017.

\bibitem{SchWa17}
K.~Schwieger and S.~Wagner.
\newblock {Part III}, {Free} actions of compact quantum groups on
  {C$^*$}-algebras.
\newblock {\em SIGMA}, 13(062):19 pages, 2017.

\end{thebibliography}

\end{document}